\documentclass[12pt,twoside]{article}
\usepackage{amsmath,amssymb}    
\usepackage{url}                
\usepackage{graphicx}           
\usepackage{amsthm}
\theoremstyle{definition}
\newtheorem{theorem}{Theorem}[section]
\newtheorem{lemma}[theorem]{Lemma}
\newtheorem{proposition}[theorem]{Proposition}
\newtheorem{corollary}[theorem]{Corollary}

\newtheorem{definition}[theorem]{Definition}

\newtheorem{question}[theorem]{Question}

\newtheorem{remark}[theorem]{Remark}

\newcommand{\HH}{\mathbb{H}}

\newcommand{\QQ}{\mathbb{Q}}
\newcommand{\ZZ}{\mathbb{Z}}

\newcommand{\OO}{\mathcal{O}}

\newcommand{\hh}{\mathfrak{h}}

\newcommand{\la}{\langle}
\newcommand{\ra}{\rangle}

\newcommand{\ttt}{\mathfrak{t}}

\date{}
\newcommand{\h}{{\mathfrak{h}}}
\newcommand{\CC}{{\Bbb C}}
\begin{document}

\pagestyle{myheadings}
\title{Unitary representations of rational Cherednik algebras}
\author{Pavel Etingof, Emanuel Stoica \\ (with an appendix by Stephen Griffeth)}
\maketitle

\begin{abstract}
We study unitarity of lowest weight irreducible
representations of rational Cherednik algebras.
We prove several general results, and use them to determine
which lowest weight representations are unitary
in a number of cases.

In particular, in type A, we give a full
description of the unitarity locus (justified in Subsection 5.1 and the appendix written by S. Griffeth),
and resolve a question by Cherednik
on the unitarity of the irreducible subrepresentation
of the polynomial representation. Also, as a by-product,
we establish Kasatani's conjecture in full generality (the previous proof
by Enomoto assumes that the parameter $c$ is not a half-integer).
\end{abstract}

\section{Introduction}
One of the important problems in the theory
of group representations is to determine when
an irreducible complex representation of a given group
is unitary. In the case of noncompact Lie groups, this is
a very hard problem, which has not been completely solved.
For p-adic groups, this problem leads to the difficult and
interesting problem of classification of unitary representations
of affine Hecke algebras.

In this paper, we begin to study the problem of classification
of unitary representations for rational Cherednik algebras.
Recall that a rational Cherednik algebra $H_c(W,\h)$
is defined by a finite group $W$, a finite dimensional complex
representation $\h$ of $W$, and a function $c$ on conjugacy
classes of reflections in $W$. Recall also that for
any irreducible representation $\tau$ of $W$, one can define the
irreducible lowest weight representation $L_c(\tau)$
of $H_c(W,\h)$. If $c(s^{-1})=\bar c(s)$ for all reflections $s$,
then the representation $L_c(\tau)$ admits a unique, up to
scaling, nondegenerate contravariant Hermitian form.
We say that $L_c(\tau)$ is unitary if this form is positive
definite (under an appropriate normalization).

The main problem is then to determine for which
$c$ and $\tau$ the representation $L_c(\tau)$ is unitary.
This problem is motivated by harmonic analysis, and was
posed by I. Cherednik. In general, it appears to be quite difficult,
like its counterpart in the theory of group representations.
The goal of this paper is to begin to attack this problem,
by proving a number of partial results about unitary representations.

More specifically, for every $\tau$ we define the set $U(\tau)$
of values of $c$ for which the representation $L_c(\tau)$ is unitary.
We prove several general results about $U(\tau)$, and
use them to determine the sets $U(\tau)$ in a number of special
cases.

In particular, Theorem \ref{genera} gives a full description of the
sets $U(\tau)$ in type A. Namely, it states that unless
$\tau$ is the trivial or sign representation (in which
case $U(\tau)=(-\infty,1/n]$ and $[-1/n,+\infty)$ respectively),
the set  $U(\tau)$ consists of the interval $[-1/\ell,1/\ell]$, where
$\ell$ is the length of the largest hook of $\tau$ (``the
continuous spectrum'') and a certain finite set of points of the form $1/j$, where $j$
are integers (``the discrete spectrum''). We note that the authors of the main body
of the paper were unable to prove Theorem \ref{genera} in its full strength;
they were only able to prove that the claimed set contains the unitarity locus, 
which in turn contains the interval $[-1/\ell,1/\ell]$, and some additional 
partial results discussed in Subection 5.1. 
The proof of Theorem \ref{genera} was completed by an argument due to S. Griffeth, which uses 
Cherednik's technique of intertwiners and Suzuki's work \cite{Su}, and is contained in
the appendix.

We also answer, for type $A$, a question by
Cherednik, proving that if $c=1/m$, $2\le m\le n$,
then the irreducible submodule $N_c$
of the polynomial module $M_c(\Bbb C)$ over the rational Cherednik algebra $H_c(S_n,\Bbb C^n)$
is unitary, and moreover its unitary structure is given
by the integration pairing with the Macdonald-Mehta measure.

As a by-product, we determine in full generality the
structure of the polynomial representation
of the rational Cherednik algebra of type $A$,
conjectured by Dunkl \cite{Du}; this implies a similar description
of the structure of the polynomial representation
of the double affine Hecke algebra, conjectured by Kasatani \cite{Ka}.
These results were established earlier by Enomoto \cite{En}
under an additional assumption that $c$ is not a half-integer,
which we show to be unnecessary.

The organization of the paper is as follows.
Section 2 contains preliminaries. In Section 3,
we prove some general properties of unitarity loci,
and completely determine them in the rank 1 case.
In Section 4, we focus on the special case of real reflection groups,
prove some general properties of the unitarity loci, and
compute them in the rank 2 case. In Section 5
we give the results in type A - prove the Dunkl-Kasatani conjecture,
answer Cherednik's question, state the theorem
on the classification of unitary representations,
and begin its proof. The proof is completed in the appendix.

{\bf Acknowledgments.}
We are very grateful to I. Cherednik,
who posed the main problem and suggested
a number of important techniques.
This paper would not have appeared
without his influence. We also thank
D. Vogan for many useful discussions
about unitary representations of Lie groups, and
Charles Dunkl for comments on a preliminary 
version of the paper.
The work of P.E. was  partially supported by the NSF grant
DMS-0504847.  The work of S.G. was partially supported
by NSF Career Grant DMS-0449102.

\section{Preliminaries}

\subsection{Definition of rational Cherednik algebras}

Let $\h$ be a finite dimensional vector space over $\CC$
with a positive definite Hermitian
\footnote{We agree that Hermitian forms are antilinear on the
second argument.} inner product $(,)$.
Let $T: \h\to \h^*$ be the antilinear isomorphism defined by the formula
$(Ty_1)(y_2)=(y_2,y_1)$.

Let $W$ be a finite subgroup
of the group of unitary transformations of $\h$.
A reflection in $W$ is an element $s\in W$ such that
${\rm rk}(s-1)|_\h=1$. Denote by $S$ the set of reflections in $W$.
Let $c: S\to \CC$ be a $W$-invariant function.
For $s\in S$, let $\alpha_s\in \h^*$ be a generator of ${\rm
Im}(s-1)|_{\h^*}$, and $\alpha_s^\vee\in \h$ be the generator of ${\rm
Im}(s-1)|_{\h}$, such that $(\alpha_s,\alpha_s^\vee)=2$.
If $W$ is generated by reflections, then we denote by $d_i,
i=1,...,\dim \h$, the degrees of the generators of $\Bbb C[\h]^W$.

\begin{definition} (see e.g. \cite{EG,E1})
The rational Cherednik algebra $H_c(W,\h)$ is the
quotient of the algebra $\CC W\ltimes T(\h\oplus \h^*)$
by the ideal generated by the relations
$$
[x,x']=0,\ [y,y']=0,\ [y,x]=(y,x)-\sum_{s\in S}
c_s(y,\alpha_s)(x,\alpha_s^\vee)s,
$$
$x,x'\in \h^*$, $y,y'\in \h$.
\end{definition}

An important role in the representation theory of rational
Cherednik algebras is played by the element
$$
\bold h=\sum_i x_iy_i+\frac{\dim
\h}{2}-\sum_{s\in S}\frac{2c_s}{1-\lambda_s} s,
$$
where $y_i$ is a basis of $\h$, $x_i$ the dual basis of $\h^*$,
and $\lambda_s$ is the nontrivial eigenvalue of $s$ in $\h^*$.
Its usefulness comes from the fact that it satisfies the
identities
\begin{equation}\label{sca}
[\bold h,x_i]=x_i, [\bold h,y_i]=-y_i.
\end{equation}

\subsection{Verma modules, irreducible modules, and the contravariant
form}

Let $\tau$ be an irreducible representation of $W$.
Denote by $M_c(\tau)$ the corresponding Verma module,
$M_c(\tau)=H_c(W,\h)\otimes_{\Bbb CW\ltimes S\h}\tau$, where
$\h$ acts on $\tau$ by zero.
Any quotient of $M_c(\tau)$ is called a lowest weight module with
lowest weight $\tau$. Denote by $L_c(\tau)$
the smallest of such modules, i.e. the
unique irreducible quotient of the module $M_c(\tau)$.
If confusion is possible, we will use the long notation
$M_c(W,\h,\tau)$, $L_c(W,\h,\tau)$ for $M_c(\tau)$, $L_c(\tau)$.

Denote by ${\mathcal O}_c(W,\h)$ the category of
$H_c(W,\h)$-modules which are finitely generated under
the action of $\Bbb C[\h]$, and locally nilpotent under the
action of $\h$. Examples of objects of this category are
$M_c(\tau)$ and $L_c(\tau)$.

It is easy to see that the element $\bold h$ acts locally
finitely on any object of ${\mathcal O}_c(W,\h)$,
with finite dimensional generalized eigenspaces.
In particular, it acts semisimply on any lowest weight module
$M$, with lowest eigenvalue
$$
h_c(\tau)=\frac{\dim
\h}{2}-\sum_{s\in S}\frac{2c_s}{1-\lambda_s} s|_\tau
$$
All other eigenvalues of $\bold h$ on $M$
are obtained by adding a nonnegative
integer to $h_c(\tau)$, and this nonnegative integer gives a
$\Bbb Z_+$-grading on $M$.

If $M\in {\mathcal O}_c(W,\h)$, then a vector $v\in M$ is called
a singular vector if $yv=0$ for any $y\in \h$. It is clear that
a lowest weight module $M$ is irreducible if and only if
it has no nonzero singular vectors of positive degree.

\subsection{Unitary representations}

Let $c^\dagger$ be the function on $S$ defined by
$c^\dagger(s)=\bar c(s^{-1})$.
Fix a $W$-invariant Hermitian form $(,)_\tau$ on $\tau$,
normalized to be positive definite.

\begin{proposition}\label{contra} (i) There exists a unique
$W$-invariant Hermitian form $\beta_{c,\tau}$ on $M_c(\tau)$
which coincides with $(,)_\tau$ in degree zero, and
satisfies the contravariance condition
$$
(yv,v')=(v,Ty\cdot v'),\ v,v'\in M_c(\tau), y\in \h.
$$

(ii) The kernel of $\beta_{c,\tau}$ coincides with the maximal proper submodule
$J_c(\tau)$ of $M_c(\tau)$, so this form descends to a
nondegenerate form on the quotient
$M_c(\tau)/J_c(\tau)=L_c(\tau)$.
\end{proposition}

\begin{proof}
Standard.
\end{proof}

We'll call $\beta_{c,\tau}$ the contravariant Hermitian form.
It is defined uniquely up to a positive scalar, which will not be
important.

Let $C$ denote the space of functions $c$ such that
$c=c^\dagger$.

\begin{definition}
Let $c\in C$. The representation $L_c(\tau)$ is said to be
unitary if the form $\beta_{c,\tau}$ is positive definite on
$L_c(\tau)$.
\end{definition}

\begin{definition}
$U(\tau)$ is the set of points $c\in C$,
such that $L_c(\tau)$ is unitary.
We call $U(\tau)$ the unitarity locus for $\tau$.
\end{definition}

\section{General properties of the sets $U(\tau)$}

\subsection{The general case}

\begin{proposition}\label{zer}
(i) $U(\tau)$ is a closed set in $C$.

(ii) The point $0$ belongs to the interior of $U(\tau)$
for any $\tau$.

(iii) The connected component $Y_0(\tau)$ of $0$ in the set
$Y(\tau)$ of all $c\in C$
for which the form $\beta_{c,\tau}$ is nondegenerate (i.e.,
$M_c(\tau)$ is irreducible) is contained in $U(\tau)$.

(iv) Let $M$ be a lowest weight representation of
$H_{c_0}(W,\h)$ which is the limit of a 1-parameter family
of irreducible unitary representations $L_{c_0+tc_1}(\tau)$,
$t\in (0,\varepsilon)$, as $t$ goes to $0$.
Then all composition factors of $M$ are unitary.
\end{proposition}

\begin{proof}
(i) $c\in U(\tau)$ iff the contravariant form
is nonnegative definite on $M_c(\tau)$, which is a closed
condition on $c$.

(ii) We have a natural identification of $M_c(\tau)$
with $\tau\otimes \Bbb C[\h]$, and
the form $\beta_{0,\tau}$ is the tensor product of the
form $(,)_\tau$ on $\tau$ and the standard inner product on $\Bbb
C[\h]$, given by the formula $(f,g)=(D_g f)(0)$,
$D_g\in S\h$ being the differential operator on $\h$ with constant coefficients
corresponding to $g\in S\h^*$ (via the operator $T$).
Thus $\beta_{0,\tau}>0$, as desired.

(iii) This follows from the standard fact that a continuous family of
nondegenerate Hermitian forms is positive definite iff one of them
is positive definite.

(iv) This follows from the standard argument with the Jantzen
filtration.
\end{proof}

It is useful to consider separately the case of constant
functions $c\in C$ (in this case $c$ is real).
Namely, let $U^*(\tau)\subset \Bbb R$ be the set
of all $c\in \Bbb R$ that belong to $U(\tau)$.
It is easy to see that analogously to
Proposition \ref{zer}, we have:

\begin{corollary}\label{zer1}
(i) $U^*(\tau)$ is a closed set in $\Bbb R$.

(ii) The point $0$ belongs to the interior of $U^*(\tau)$
for any $\tau$.

(iii) The connected component $Y_0^*(\tau)$ of $0$ in the set
$Y^*(\tau)$ of all $c\in \Bbb R$
for which the form $\beta_{c,\tau}$ is nondegenerate (i.e.,
$M_c(\tau)$ is irreducible) is contained in $U^*(\tau)$.
\end{corollary}

Let $W_{\rm ab}^\vee$ be the group of characters of $W$.
It is easy to see that $W_{\rm ab}^\vee$ acts on the space $C$
by multiplication. It also acts on representations of $W$ by
tensor multiplication.

\begin{proposition}\label{char}
For any $\chi\in W_{\rm ab}^\vee$ one has $U(\chi\otimes \tau)=\chi U(\tau)$.
\end{proposition}

\begin{proof}
The statement follows from the fact that we have a natural
isomorphism $i_\chi: H_c(W,\h)\to H_{\chi^{-1} c}(W,\h)$ given by the
formula $w\to \chi^{-1}(w)w$, $w\in W$, and $i_\chi(x)=x$,
$i_\chi(y)=y$, $x\in \h^*$, $y\in \h$. The pushforward by this isomorphism
maps $\tau$ to $\chi\otimes \tau$, which implies the statement.
\end{proof}

\begin{proposition}\label{deg1con}
If $c\in U(\tau)$ then for any irreducible representation
$\sigma$ of $W$ that occurs in $\tau\otimes \h^*$,
one has $h_c(\sigma)\le h_c(\tau)+1$.
\end{proposition}

\begin{proof}
We will need the following easy lemma (which is probably known,
but we give its proof for reader's convenience).

\begin{lemma}\label{deg1}
Let $\sigma\subset \tau\otimes \h^*$ be an irreducible subrepresentation.
Let us regard $\sigma$ as sitting in degree $1$ of $M_c(\tau)$.
Then the elements $y_i$ act on $\sigma$ by $0$ (i.e. $\sigma$
consists of singular vectors) if and only if
$h_c(\sigma)-h_c(\tau)=1$.
\end{lemma}

\begin{proof} The action of $y_i$ on the degree 1 part of
$M_c(\tau)$ can be viewed as an operator $\tau\otimes \h^*\otimes
\h\to \tau$, or, equivalently, as an endomorphism $F_{c,\tau,1}$ of
$\tau\otimes \h^*$. This endomorphism is easy to compute, and it
is given by the formula
\begin{equation}\label{Fc1}
F_{c,\tau,1}=1-\sum_{s\in S}c_s s\otimes (\alpha_s\otimes \alpha_s^\vee)=
1-\sum_{s\in S}\frac{2c_s}{1-\lambda_s} s\otimes (1-s).
\end{equation}
Thus $F_{c,\tau,1}$ acts on $\sigma$ by the scalar
$$
1+h_c(\tau)-h_c(\sigma).
$$
The action of $y_i$ on $\sigma$ is zero iff this scalar is zero,
which implies the lemma.
\end{proof}

Now look at the restriction of the form $\beta_{c,\tau}$ to an irreducible
$W$-subrepresentation $\sigma$ sitting in the degree 1 part
$\tau\otimes \h^*$ of $M_c(\tau)$. This restriction is
obviously of the form $\ell(c)(,)_\sigma$, where
$\ell(c)$ is a linear nonhomogeneous function of $c$.
Since $\ell(c)$ is positive for $c=0$ (by Proposition
\ref{zer}(ii)), we conclude, using Lemma
\ref{deg1}, that $\ell(c)=K(1+h_c(\tau)-h_c(\sigma))$,
where $K>0$. This implies the statement.
\end{proof}

Let $D_\tau$ be the eigenvalue of $\sum_{s\in S}s$ on
$\tau$.

\begin{corollary}\label{deg1con1}
If $c\in U^*(\tau)$ then for any $\sigma$ contained in
$\tau\otimes \h^*$, one has
$$
c(D_\tau-D_\sigma)\le 1.
$$
\end{corollary}

\subsection{The operator $F_{c,\tau,m}$}

It is useful to generalize the operator $F_{c,\tau,1}$ acting in
degree $1$ to higher degrees. Namely,
for $c\in C$, we have a unique selfadjoint operator $F_{c,\tau}$ on
$M_c(\tau)=\tau \otimes S\h^*$,
given by the formula
$\beta_{c,\tau}(v,v')=\beta_{0,\tau}(F_{c,\tau}v,v')$.
We have $F_{c,\tau}=\oplus_{m\ge 0}F_{c,\tau,m}$, where $F_{c,\tau,m}:
\tau\otimes S^m\h^*
\to \tau\otimes S^m\h^*$ is an operator which is polynomial in $c$
of degree at most $m$.
It is clear that if $F_{c,\tau,m}$ is independent of $c$, then
$F_{c,\tau,m}=1$, because $F_{0,\tau,m}=1$.
Also, we have the following recursive formula for $F_{c,\tau,m}$:

\begin{proposition}\label{recfor}
Let $a_1,...,a_m\in \h^*$, and $v\in \tau$.
Then
$$
F_{c,\tau,m}(a_1...a_mv)=
$$
$$
\frac{1}{m}\sum_{j=1}^ma_jF_{c,\tau,m-1}(a_1...a_{j-1}a_{j+1}...a_mv)-
$$
$$
-\frac{1}{m}\sum_{j=1}^m\sum_{s\in
S}\frac{2c_s}{1-\lambda_s}(1-s)(a_j)F_{c,\tau,m-1}(a_1...a_{j-1}s(a_{j+1}...a_mv))
$$
\end{proposition}

\begin{remark} Note that for $m=1$ this formula reduces to
formula (\ref{Fc1}).
\end{remark}

\begin{proof}
It is easy to see that for any $y\in \h$, one has
$$
F_{c,\tau,m-1}(ya_1...a_mv)=\partial_y F_{c,\tau,m}(a_1...a_mv).
$$
Therefore, we find
$$
F_{c,\tau,m}(u)=\frac{1}{m}\sum_i x_iF_{c,\tau,m-1}(y_iu).
$$
Taking $u=a_1...a_mv$ and computing $y_iu$ using the commutation
relations of the rational Cherednik algebra, we get the result.
\end{proof}

\begin{corollary}\label{lin}
Suppose that $F_{c,\tau,i}$ is constant (and hence equals $1$)
for $i=1,...,m-1$. Then on every irreducible
$W$-subrepresentation $\sigma$ of $\tau\otimes S^m\h^*$,
the operator $F_{c,\tau,m}$ acts by the scalar
$1+\frac{h_c(\tau)-h_c(\sigma)}{m}$.
\end{corollary}

\subsection{The rank 1 case}
Suppose $\h$ is 1-dimensional, and $W=\Bbb Z/m\Bbb Z$, acting
by $j\to \lambda^{-j}$, where $\lambda$ is a primitive $m$-th
root of unity. In this case all irreducible
representations of $W$ are 1-dimensional, so thanks to
Proposition \ref{char}, to describe the sets $U(\tau)$, it
suffices to describe the set $U:=U(\Bbb C)$ for the trivial
representation $\Bbb C$. Let us find $U$.

The module $M_c(\Bbb C)$ has basis $x^n$, $n\ge 0$.
Let $a_n:=\beta_{c,\Bbb C}(x^n,x^n)$ (we normalize the form so
that $a_0=1$). It is easy to compute that
$$
a_n=a_{n-1}(n-2\sum_{j=1}^{m-1}\frac{1-\lambda^{jn}}{1-\lambda^j}c_j),
$$
where $c_j=c(j)$, $j=1,...,m-1$.

Let
$$
b_n:=2\sum_{j=1}^{m-1}\frac{1-\lambda^{jn}}{1-\lambda^j}c_j,
$$
$n\ge 0$ (note that $b_0=0$, and $b_{n+m}=b_n$).
If $c\in C$ then $b_j$ are real, and
it is easy to see that $b_1,...,b_{m-1}$
form a linear real coordinate system on
$C$ (this follows from the easy fact that the matrix with entries
$\frac{1-\lambda^{jn}}{1-\lambda^j}$, $1\le j,n\le m-1$,
is nondegenerate).

This implies the following proposition.

\begin{proposition}
(i) $M_c(\Bbb C)$ is irreducible iff $n-b_n\ne 0$
for any $n\ge 1$. It is unitary iff $n-b_n>0$ for all
$n=1,...,m-1$.

(ii) Assume that for a given $c$, $r$ is the smallest positive integer
such that $r=b_r$. Then $L_c(\Bbb C)$ has dimension $r$ (which
can be any number not divisible by $m$), and basis
$1,x,...,x^{r-1}$. This representation can be unitary only if
$r<m$, and in this case it is unitary iff $n-b_n>0$ for $n<r$.
\end{proposition}

\begin{corollary}
$U$ is the set of all vectors $(b_1,...,b_{m-1})$ such that
in the vector $(1-b_1,2-b_2,...,m-1-b_{m-1})$, all the entries
preceding the first zero entry are positive (if there is no
zero entries, all entries must be positive).
\end{corollary}

In particular, if $m=2$ and $c_1=c$, then $b_1=2c$, and we find
that $U=(-\infty,1/2]$ (at the point $c=1/2$ the unitary
representation is 1-dimensional).

\section{The real reflection case}

\subsection{The sl(2) condition}

In the rest of the paper, we'll assume that
$\h$ is the complexification of a real vector space $\h_{\Bbb R}$
with a positive definite inner product, which is extended to a
Hermitian inner product on the complexification, and
that $W$ acts by orthogonal transformations on $\h_{\Bbb R}$.
Then $s^2=1$ for any reflection $s$, and thus $c\in C$
iff $c$ is real valued.

In this case, let us choose $y_i$ to be an orthonormal
basis of $\h_{\Bbb R}$. Then
it is easy to see that
$$
\bold h=\frac{1}{2}\sum(x_iy_i+y_ix_i),
$$
and we also have elements
$$
\bold e=-\frac{1}{2}\sum x_i^2, \bold f=
\frac{1}{2}\sum y_i^2,
$$
These elements form an ${\frak {sl}}_2$-triple.

\begin{proposition}\label{sl2}
(i) A unitary representation $L_c(\tau)$ of $H_c(W,\h)$ restricts to
a unitary representation of ${\frak {sl}}_2(\Bbb R)$ from lowest
weight category ${\mathcal O}$. In particular, $h_c(\tau)
=\frac{\dim \h}{2}-\sum c_s s|_\tau\ge 0$.

(ii) A unitary representation
$L_c(\tau)$ is finite dimensional iff $L_c(\tau)=\tau$;

(iii) An irreducible lowest weight
representation $L_c(\tau)$ coincides with $\tau$
iff $h_c(\sigma)-h_c(\tau)=1$
for any irreducible representation $\sigma$ of $W$
contained in $\tau\otimes \h^*$.
In this case $h_c(\tau)=0$.
\end{proposition}

\begin{proof}
(i) Straightforward.

(ii) If $L_c(\tau)$ is finite dimensional, then by (i), it is a trivial
representation of ${\frak {sl}}_2(\Bbb R)$. So $\bold h=0$,
and hence by (\ref{sca}), $x_i=0$, which implies the statement.

(iii) The statement $L_c(\tau)=\tau$ is equivalent
to the statement that $y_i$ acts by $0$ on any subrepresentation
$\sigma$ in $\tau\otimes \h^*$,
which by Lemma \ref{deg1} is equivalent the condition
$h_c(\sigma)-h_c(\tau)=1$.
\end{proof}

\subsection{Unitarity locus $U^*(\tau)$
for exterior powers of the reflection
representation}

Let $W$ be an irreducible Coxeter group, and
$\h$ be its reflection representation.
Recall that the representations $\wedge^i\h$ are irreducible.
In particular, $\wedge^{\dim \h}\h$ is the sign representation
$\Bbb C_-$ of $W$.

\begin{corollary}\label{coxe}
(i) For all $\tau$ one has
$U^*(\tau)\supset [-1/h,1/h]$, where $h$ is the Coxeter
number of $W$;

(ii) $U^*(\Bbb C)=(-\infty,1/h]$, and $U^*(\Bbb
C_-)=[-1/h,+\infty)$;

(iii) For $0<i<\dim \h$, $U^*(\wedge^i \h)=[-1/h,1/h]$.
\end{corollary}

\begin{proof}
(i) It is known (\cite{DJO,GGOR}) that if $c\in (-1/h,1/h)$ then $c$ is a regular
value, which means that the category ${\mathcal O}_c(W,\h)$ is
semisimple. So all $M_c(\tau)$ are irreducible, which implies
the desired statement by Proposition \ref{zer}(iii).

(ii) Suppose $c\in U^*(\Bbb C)$.
We have $h_c(\Bbb C)=\frac{\dim \h}{2}-c|S|=|S|(\frac{1}{h}-c)$.
Since $h_c(\Bbb C)\ge 0$, we get $c\le 1/h$. On the other hand,
for $c<0$ the module $M_c(\Bbb C)$ is irreducible, hence
unitary. So the first statement of (ii) follows from (i).
The second statement of (ii) follows from the first one
by Proposition \ref{char}.

(iii) The ``$\supset$'' part follows from (i). To prove the ``$\subset$''
part, note that the irreducible representation $\wedge^{i+1}\h$
sits naturally in the degree $1$ part of $M_c(\wedge^i\h)$.

Let us compute $D_{\wedge^i\h}$.
It is easy to see that the trace of a reflection in $\wedge^i\h$ is
$$
\binom{\dim \h-1}{i}-\binom{\dim \h-1}{i-1}.
$$
Thus, we have
$$
D_{\wedge^i\h}=|S|\frac{\binom{\dim \h-1}{i}-
\binom{\dim \h-1}{i-1}}
{\binom{\dim \h}{i}}=|S|(1-\frac{2i}{\dim \h}).
$$
Hence,
$$
h_c(\wedge^{i+1}\h)-h_c(\wedge^i\h)=
2c|S|/\dim \h=ch,
$$
So proposition \ref{deg1con} tells us that
for any $c\in U^*(\wedge^i\h)$ one has $ch\le 1$.
The rest follows from Proposition \ref{char}
and part (i).
\end{proof}

\subsection{The rank 2 case}

In this subsection we will calculate the sets $U(\tau)$ in the
rank 2 case, i.e., for dihedral groups $W$.

We start with odd dihedral groups.
Let $W$ be the dihedral group whose order is $2(2d+1)$. This group has
only one conjugacy class of reflections (so $C=\Bbb R$),
two 1-dimensional representations, $\Bbb C$ and $\Bbb C_-$, and $d$
2-dimensional representations $\tau_l$, $l=1,...,d$, defined by
the condition that the counterclockwise rotation by the angle $2\pi/(2d+1)$
acts in this representation with eigenvalues $\zeta^l$ and
$\bar\zeta^l$, where $\zeta=e^{\frac{2\pi i}{2d+1}}$.
The reflection representation $\h$ is thus the representation $\tau_1$.

\begin{proposition}
$U(\Bbb C)=(-\infty,\frac{1}{2d+1}]$,
$U(\Bbb C_-)=[-\frac{1}{2d+1},+\infty)$, and
$U(\tau_l)=[-\frac{l}{2d+1},
\frac{l}{2d+1}]$ for all $1\leq l\leq d$.
\end{proposition}

\begin{proof}
The first two statements are special cases of
Corollary \ref{coxe}(ii), since the Coxeter number $h$
of $W$ is $2d+1$. To prove the last statement, let us
look at the decomposition $S^{k}\tau_1=\tau_{k}\oplus
\tau_{k-2}\oplus...$ (the last summand is $\Bbb C$ if $k$ is
even). By tensoring this decomposition
with $\tau_l$, we notice that we obtain only 2-dimensional
summands if $k<l$, while one-dimensional summands
make their first appearance only for $k=l$. It follows
by induction in $k$, using Corollary \ref{lin},
that the operator $F_{c,\tau_l,k}$ is constant in $c$
(and hence equal to $1$) for $k<l$ (as $h_c(\tau)=1$ for any
two-dimensional $\tau$). Thus, again by Corollary
\ref{lin}, $F_{c,\tau_l,l}(X)=(1\pm \frac{2d+1}{l}c)X$
if $X$ belongs to the sign, respectively trivial
subrepresentation of $\tau_l\otimes S^l\tau_1$.
This implies that if $c\in U(\tau_l)$,
then we must have $c\in [-\frac{l}{2d+1},\frac{l}{2d+1}]$.

It remains to show that $M_c(\tau_l)$ is irreducible if
$(2d+1)|c|<l$. This is proved in the paper \cite{Chm}, and can
also be proved directly, as follows. It follows from the above
that $M_c(\tau_l)$ contains no singular vectors
of degree $<l$. Assume that $c>0$; then any singular vector
would be in the sign representation. Let $k\ge l$ be the degree
of this vector. Then we get $h_c(\Bbb C_-)-h_c(\tau_l)=k$, which
implies that $(2d+1)c=k\ge l$, as desired. The case of negative
$c$ is similar.
\end{proof}

Let us now analyse the case of even dihedral group
$W$, of order $4d$, $d \geq 2$ (the dihedral
group of a regular $2d$-polygon).
In this case there are two conjugacy classes of reflections,
represented by Coxeter generators $s_1,s_2$, such that
$(s_1s_2)^{2d}=1$. The 1-dimensional
representations of $W$ are $\Bbb C$, $\Bbb C_-$, and also the
representations $\varepsilon_1$ and
$\varepsilon_2$, defined by the formulas

\begin{displaymath}
\varepsilon_1 \colon \left\{ \begin{array}{c} s_1 \longrightarrow -1 \\
s_2 \longrightarrow 1  \end{array} \right.
\varepsilon_2 \colon \left\{ \begin{array}{c} s_1 \longrightarrow 1 \\
s_2 \longrightarrow -1  \end{array} \right.
\end{displaymath}

In addition, there are $d-1$ 2-dimensional representations $\tau_l$,
for all $1\leq l \leq d-1 $, given by the same formulas as in the odd
case; in particular, as before, $\h=\tau_1$.
We will extend the notation $\tau_l$ to all integer values of $l$,
so that we have $\tau_l=\tau_{-l}$ and
$\tau_{d-l}=\tau_{d+l}$, $\tau_0=\mathbb{C} \oplus
\mathbb{C}_-$, and $\tau_d=\varepsilon_1 \oplus \varepsilon_2$.
Note that $\tau_l \otimes \varepsilon_i=\tau_{d-l}$ and $\tau_l \otimes
\Bbb C_-=\tau_l$.

Let $c_1$ and $c_2$ be the values of the parameter $c$ on the two
conjugacy classes of reflections. We will now describe the sets
$U(\tau)$ in the plane with coordinates $c_1,c_2$. By
Proposition \ref{char}, it suffices to find $U(\tau)$ for
$\tau=\Bbb C$ and $\tau=\tau_l$, $1\le l\le d-1$.

\begin{proposition}
(i) $U(\mathbb{C})$ is the union of the region defined by the inequalities
$c_1+c_2 < \frac{1}{d}$, $c_1 \leq \frac{1}{2}$ and $c_2 \leq
\frac{1}{2}$ with the line $c_1+c_2=\frac{1}{d}$.

(ii) If $1\le l\le d-1$ then
$U(\tau_l)$ is the rectangle defined by the inequalities
$|c_1+c_2|\le \frac{l}{d}$ and $|c_1-c_2|
\le \frac{d-l}{d}$.
\end{proposition}

\begin{proof}
(i) The operator $F_{c,\Bbb C,1}$ is the scalar
$1-(c_1+c_2)d$. This implies the condition $c_1+c_2\leq
\frac{1}{d}$ for $c\in U(\Bbb C)$.

Now recall that
$S^k \tau_1=\tau_k \oplus \tau_{k-2} \oplus\cdots$
(the last summand is $\Bbb C$ for even $k$).
In particular, $S^d\tau_1$ contains $\varepsilon_1$ and
$\varepsilon_2$, one copy of each. Consider the operator $F_{c,\Bbb C}$
restricted to the subrepresentation $\varepsilon_i$.
We claim that this operator (which is a scalar, since it is defined on a
one-dimensional space) equals
\begin{equation}\label{prodd}
Q(c)=(1-2c_i)\prod_{j=1}^{d-1}(1-\frac{d}{j}(c_1+c_2)).
\end{equation}
Indeed, it follows from the paper \cite{Chm}
that $\varepsilon_i$ consists of singular vectors if
$c_i=1/2$, and that at the line $c_1+c_2=\frac{l}{d}$, $1\le l\le
d-1$, there is a singular vector of degree $l$ in the representation
$\tau_l$, such that the subrepresentation generated by this
vector contains $\varepsilon_i$ in degree $d$. This implies that
$Q(c)$ is divisible by the right hand side of (\ref{prodd}).
On the other hand, the degree of $Q(c)$ is $d$, and $Q(0)=1$,
which implies (\ref{prodd}).

Formula  (\ref{prodd}) and the inequality $c_1+c_2\le 1/d$
implies that if a unitary representation $L_c(\Bbb C)$ contains
$\varepsilon_i$ in degree $d$, then we must have $c_i<1/2$.
It remains to consider unitary representations $L_c(\Bbb C)$ that do not
contain $\varepsilon_i$ for some $i$. This means that either
this $\varepsilon_i$ is singular in $M_c(\Bbb C)$ (which means $c_i=1/2$) or
the copy of $\tau_1$ in degree 1 is singular,
i.e. $c_1+c_2=1/d$. This proves part (i).

(ii) Assume that $l<d/2$. Similarly to the case of odd dihedral group,
there is no 1-dimensional representations in $M_c(\tau_l)$ in
degrees $k<l$, while the trivial and sign representations sit in
degree $l$. As in the odd case, this implies,
by using induction in $k$ and Corollary \ref{lin} that
$F_{c,\tau_l,i}=1$ for $i<l$, and $F_{c,\tau_l,l}(X)=(1\pm
\frac{d}{l}(c_1+c_2))X$
if $X$ belongs to the sign, respectively trivial
subrepresentation of $\tau_l\otimes S^l\tau_1$.
This implies that if $c\in U(\tau)$ then
$|c_1+c_2| \leq\frac{l}{d}$.

Let us now prove the second inequality
$|c_1-c_2|\le \frac{d-l}{d}$. By \cite{Chm}, we have singular
vectors living in $\varepsilon_1$ and $\varepsilon_2$ in degree $d-l$.
Since $S^{d-2l}\tau_1$ does not contain 1-dimensional
representations, by using the same argument as in the proof of
Proposition \ref{lin}, we conclude that $F_{c,\tau_l,d-l}$ acts
on $\varepsilon_i$ by the scalars \linebreak $1\pm \frac{d}{d-l}(c_1-c_2)$,
which proves the desired inequality for unitary representations.

Finally, if both inequalities are satisfied strictly, then it
follows from \cite{Chm} that $M_c(\tau_l)$ is irreducible,
and thus the rectangle defined by these inequalities is contained
in $U(\tau_l)$, as desired.

If $l\ge d/2$, the result is obtained by applying Proposition
\ref{char} to $\chi=\varepsilon_1$.
Part (ii) is proved.
\end{proof}

\subsection{The Gaussian inner product}

Part (ii) of Proposition \ref{sl2}
can be generalized. For this purpose we want
to introduce the Gaussian inner product on any lowest weight
representation $M$ of $H_c(W,\h)$, which was defined by
Cherednik (\cite{Ch1}).

\begin{definition}
The Gaussian inner product
$\gamma_{c,\tau}$ on $M_c(\tau)$
is given by the formula
$$
\gamma_{c,\tau}(v,v')=\beta_{c,\tau}(\exp(\bold f)v,\exp(\bold
f)v').
$$
\end{definition}

This makes sense because the operator $\bold f$ is locally nilpotent on
$M_c(\tau)$. Thus we see that $\gamma_{c,\tau}$ has kernel
$J_c(\tau)$, so it descends to an inner product on any lowest
weight module with lowest weight $\tau$, in particular to
a nondegenerate inner product on $L_c(\tau)$,
and it is positive definite on $L_c(\tau)$ if and only if
so is $\beta_{c,\tau}$.
The difference between $\beta$ and $\gamma$ is that
vectors of different degrees are orthogonal under $\beta$, but
not necessarily under $\gamma$.

\begin{proposition}\label{x-inv}
(i) The form $\gamma_{c,\tau}$ on a lowest weight module $M$
satisfies the condition
$$
\gamma_{c,\tau}(xv,v')=\gamma_{c,\tau}(v,xv'),\ x\in
\h_{\Bbb R}^*.
$$

(ii) Up to scaling, $\gamma_{c,\tau}$
is the unique $W$-invariant form satisfying
the condition
$$
\gamma_{c,\tau}((-y+Ty)v,v')=\gamma_{c,\tau}(v,yv'),\ y\in
\h_{\Bbb R}.
$$
\end{proposition}

\begin{proof} (i) We have
$$
\gamma_{c,\tau}(xv,v')=
\beta_{c,\tau}(\exp(\bold f)xv,\exp(\bold
f)v')=
$$
$$
\beta_{c,\tau}((x+T^{-1}x)\exp(\bold f)v,\exp(\bold
f)v')=
\beta_{c,\tau}(\exp(\bold f)v,(T^{-1}x+x)\exp(\bold
f)v')=
$$
$$
\beta_{c,\tau}(\exp(\bold f)v,\exp(\bold
f)xv')=
\gamma_{c,\tau}(v,xv').
$$

(ii) A similar computation to (i) yields that the required
property holds. Let us now show uniqueness.
If $\gamma$ is any $W$-invariant
Hermitian form satisfying the condition of (ii), then
let $\beta(v,v')=\gamma(\exp(-\bold f)v,\exp(-\bold f)v')$.
Then $\beta$ is contravariant, so by Proposition \ref{contra},
it's a multiple of $\beta_{c,\tau}$, hence $\gamma$ is a multiple of
$\gamma_{c,\tau}$.
\end{proof}

\begin{corollary}\label{rad}
Let $c\in U(\tau)$, and let $I_c(\tau)\subset\Bbb C[\h]$
be the annihilator of $L_c(\tau)$ in $\Bbb C[\h]$. Then
$I_c(\tau)$ is a radical ideal.
\end{corollary}

\begin{proof}
Assume that $g^2\in I_c(\tau)$.
Then $g^2\bar g^2\in I_c(\tau)$, so for any $v\in L_c(\tau)$,
$\gamma_{c,\tau}(g^2\bar g^2v,v)=0$.
So by Proposition \ref{x-inv},
$\gamma_{c,\tau}(g\bar gv,g\bar gv)=0$.
Hence $g\bar gv=0$, so
$\gamma_{c,\tau}(g\bar gv,v)=0$,
i.e. $\gamma_{c,\tau}(gv,gv)=0$, which implies that $gv=0$, hence
$g\in I_c(\tau)$.
\end{proof}

This corollary is clearly a generalization of Proposition \ref{sl2}(ii).

\begin{corollary} \label{num}
Let $c$ be a constant function, and $c\in U(\tau)$.
If $I_c(\tau)\ne 0$ (i.e. the support of $L_c(\tau)$ is not equal
$\h$, and has smaller dimension), then $c=1/m$, where $m$ is an integer.
\end{corollary}

\begin{proof}
We will use the results of \cite{BE}.
Consider the support $X\subset \h$ of $L_c(\tau)$.
By our assumption, $X\ne \h$. Let $b\in X$ be a generic point.
Consider the restriction $N={\rm Res}_b(L_c(\tau))$ defined
in \cite{BE}. Then $N$ is a finite dimensional irreducible module over
$H_c(W_b,\h/\h^{W_b})$. Moreover, by Corollary \ref{rad},
$I_c(\tau)$ is a radical ideal, which implies
that $N=N_c(\xi)=\xi$ for some irreducible module $\xi$ of $W_b$.
Therefore, we have $c(D_\sigma-D_\xi)=1$, where
$D_\psi$ is the eigenvalue of $\sum_{s\in S\cap W_b}s$ on
an irreducible representation $\psi$ of $W_b$, and
$\sigma$ is an irreducible subrepresentation of $\xi\otimes (\h/\h^{W_b})$.
This implies that the numerator of $c$ is $1$, as desired.
\end{proof}

\subsection{Integral representation of the Gaussian inner product
on $M_c(\Bbb C)$.}

We will need the following known result (see \cite{Du2}, Theorem 3.10).

\begin{proposition}\label{inte}
We have
\begin{equation}\label{intformu}
\gamma_{c,\Bbb C}(f,g)=K(c)^{-1}\int_{\h_{\Bbb R}}f(z)\overline{g(z)}d\mu_c(z)
\end{equation}
where
$$
d\mu_c(z):=e^{-|z|^2/2}\prod_{s\in S}|\alpha_s(z)|^{-2c_s}dz,
$$
and
\begin{equation}\label{kc}
K(c)=\int_{\h_{\Bbb R}}d\mu_c(z),
\end{equation}
provided that the integral (\ref{kc}) is absolutely convergent.
\end{proposition}

\begin{proof}
It follows from Proposition \ref{x-inv} that $\gamma_{c,\tau}$
is uniquely, up to scaling,
determined by the condition that it is $W$-invariant,
and $y_i^\dagger=x_i-y_i$. These properties
are easy to check for the right hand side of (\ref{intformu}), using the fact that
the action of $y_i$ is given by Dunkl operators.
\end{proof}

\begin{remark}
As usual, the integral formula extends
analytically to arbitrary complex $c$.
\end{remark}

\begin{remark}
The constant $K(c)$ is given by the Macdonald-Mehta product
formula, proved by E. Opdam \cite{Op} for Weyl groups and by F. Garvan for $H_3$ and $H_4$
(for the dihedral groups the formula follows from the beta integral). For an irreducible
reflection group $W$ and a constant $c$, this
formula has the form
$$
K(c)=K_0\prod_{j=1}^{\dim \h}\frac{\Gamma(1-d_jc)}{\Gamma(1-c)},
$$
where $d_j$ are the degrees of $W$. It follows from this
formula that for constant $c$ the first pole of
$K(c)$ occurs at $c=1/h$, which gives another proof of
Corollary \ref{coxe}(ii).
\end{remark}

\subsection{The simple submodule of the polynomial representation}

Let $c$ be a constant positive function,
which is a singular value for $W$ (i.e., it is rational
and has denominator dividing one of the degrees $d_i$ of $W$).
Let $N_c$ be the minimal nonzero submodule of
the polynomial representation $M_c(\Bbb C)$. This is an
irreducible module of the form $L_c(\tau_c)$, where $\tau_c$ is
a certain irreducible representation of $W$ depending on $c$.
It is easy to see that $N_c$ is the unique simple submodule of the polynomial
representation.

The following observation was made by I. Cherednik.

\begin{proposition}\label{che1}
Suppose that $N_c$ is contained in $L^2(\h_{\Bbb R},d\mu_c)$.
Then $N_c$ is a unitary representation.
\end{proposition}

\begin{proof}
Like in the proof of Proposition \ref{inte}, we see that
the integral gives a $W$-invariant form $\gamma$ on $N_c$ that satisfies
the condition $y_i^\dagger=x_i-y_i$. By Proposition \ref{x-inv},
this implies that $\gamma$ is a multiple of $\gamma_{c,\tau_c}$.
Also, it is manifestly positive definite, as desired.
\end{proof}

Motivated by this observation and a number of examples,
Cherednik asked the following question.

\begin{question}\label{che2} (\cite{Ch1})
Let $W$ be an irreducible Coxeter group, $\h$ its reflection
representation, and $c=1/d_i$.
Is it true that $N_c$ is contained in $L^2(\h_{\Bbb R},d\mu_c)$?
In particular, is $N_c$ unitary?
\end{question}

In the next section we will show that the answer to both
questions is ``yes'' in type A.

\section{Type A}

\subsection{The main theorem}\label{s1}

In this section we restrict ourselves to the case $W=S_n$, $n\ge
2$, and $\h=\Bbb C^n$. In this case we have only one conjugacy class of
reflections, so $C=\Bbb R$. Irreducible representations
$\tau$ of $S_n$ are labeled by Young diagrams (=partitions);
for instance, the trivial representation is $(n)$ (the 1-row diagram)
and the sign representation is $(1^n)$ (the 1-column diagram).
We will denote the conjugate partition to a partition $\tau$ by
$\tau^*$; the corresponding operation on representations is
tensoring with the sign representation. Abusing notation, we will
denote partitions, Young diagrams, and representations of $S_n$
by the same letter (say, $\tau$).

We let $\ell(\tau)$ be the length of the largest hook of the Young
diagram $\tau$, $m_*(\tau)$ denote the multiplicity of the largest
part of $\tau$, and set
$$
N(\tau)=\ell(\tau)-m_*(\tau)+1.
$$

The eigenvalue $D_\tau$ of $\sum_{s\in S}s$ on $\tau$
equals the content ${\rm ct}(\tau)$
of the Young diagram $\tau$, i.e. the sum of
the numbers $i-j$ over the cells $(i,j)$ of the diagram.
Therefore,
$$
h_c(\tau)=\frac{n}{2}-c\cdot {\rm ct}(\tau).
$$

\begin{proposition}\label{bound}
For a partition $\tau\ne (1^n)$ and each $c\in U(\tau)$, $c\le
\frac{1}{N(\tau)}$.
\end{proposition}

\begin{proof}
Recall that $\tau\otimes \h^*$ is the sum of representations
corresponding to Young diagrams $\lambda$ obtained from $\tau$ by removing and
adding a corner cell. Also, let $\nu(\tau)$ be the
number of parts of $\tau$. Then it is easy to see that
$N(\tau)$ is the largest value of
$\nu(\tau)+i-j$ over all
corner cells $(i,j)$ of the Young diagram $\tau$
(i.e. cells for which neither $(i,j+1)$ nor $(i+1,j)$
belong to $\tau$). Therefore, the proposition follows
from Corollary \ref{deg1con1}.
\end{proof}

\begin{proposition}\label{larghook}
 The interval $[-\frac{1}{\ell(\tau)},
\frac{1}{\ell(\tau)}]$ is contained in $U(\tau)$.
\end{proposition}

\begin{proof}
Let $q=e^{2\pi ic}$, ${\mathcal H}_n(q)$
be the Hecke algebra of $S_n$ with parameter $q$,
and $S^\tau$ be the Specht module over ${\mathcal H}_n(q)$
corresponding to $\tau$, defined in \cite{DJ1}.
Then it follows from \cite{DJ2}, Theorem 4.11, that
$S^\tau$ is irreducible if $c\in
(-\frac{1}{\ell(\tau)}, \frac{1}{\ell(\tau)})$.
By the theory of KZ functor, \cite{GGOR},
this implies that $M_c(\tau)$ is irreducible in this range.
This implies the required statement.
\end{proof}

\begin{corollary}\label{tails}
If $\tau$ and $\tau^*$ contain a part equal to 1, then
$U(\tau)=[-\frac{1}{\ell(\tau)},
\frac{1}{\ell(\tau)}]$.
\end{corollary}

\begin{proof}
This follows from Propositions \ref{bound} and \ref{larghook},
since if $\tau^*$ contains a part equal 1 then
$N(\tau)=\ell(\tau)$.
\end{proof}

\begin{proposition}\label{thinuni}
Let $\tau=(p,p,...,p)$, where $p$ is a divisor of $n$.
Then $L_{1/p}(\tau)$ is unitary.
\end{proposition}

\begin{proof} This is shown in the proof of Theorem 8.8 in \cite{CEE}.
\end{proof}

The main result of this subsection is the following theorem.

\begin{theorem}\label{genera}
For any $\tau\ne (n),(1^n)$, $U(\tau)$ is the union of the interval
$[-\frac{1}{\ell(\tau)},\frac{1}{\ell(\tau)}]$
with the finite set of isolated points
$\frac{1}{k}$, for $N(\tau)\le k<l(\tau)$
and $-\ell(\tau^*)<k \le -N(\tau^*)$
(so there are $m_*(\tau)-1$ positive points, and $m_*(\tau^*)-1$
negative points).
\end{theorem}

The proof of Theorem \ref{genera} is begun in this subsection, and
finished in the appendix.

\begin{theorem}\label{conta}
For any $\tau\ne (n),(1^n)$, $U(\tau)$ is contained in the set
defined in Theorem \ref{genera}.
\end{theorem}

\begin{proof} Let $N=N(\tau)$, $\ell=\ell(\tau)$, $m_*=m_*(\tau)$
(so $\ell=N+m_*-1$).
By Propositions \ref{bound} and \ref{larghook}, our job is to show that the intervals
$I_k=(\frac{1}{N+k},\frac{1}{N+k-1})$, $k=1,...,m_*-1$
do not intersect with $U(\tau)$.

Denote by $\tau_i, i=1,...,m_*$
the partition of $n$ obtained by reducing $i$ copies of the largest part
of $\tau$ by $1$, and adding $i$ copies of the part $1$.
Then it follows from the rule of tensoring by $\h^*$ that
$\tau\otimes S^i\h^*$ contains a unique copy of $\tau_i$.
This implies that for any $c$, $M_c(\tau)$ contains a unique
copy of $\tau_i$ in degree $i$. We have
$\beta_{c,\tau}|_{\tau_i}=f_{i,\tau}(c)(,)_{\tau_i}$, where $f$
is a scalar polynomial.

\begin{lemma}\label{fprod}
One has, up to scaling:
$$
f_{\tau,i}(c)=(1-(N+i-1)c)...(1-Nc).
$$
\end{lemma}

\begin{proof}
The proof is by induction in $i$. For the base we can take the
case $i=0$, which is trivial. To make the inductive step, assume
that the statement is proved for $i=m-1$ and let us prove it for
$i=m$. By the induction assumption, at $c=\frac{1}{N+j-1}$,
$j=1,...,m-1$, the module $M_c(\tau)$ has a singular vector
$u$ sitting in $\tau_j$ in degree $j$. Indeed, the contravariant
form on $\tau_j$ is zero at such $c$, and there can be no
singular vectors of lower degree, because
if one moves $i<j$ corner cells of $\tau$ to get a partition
$\sigma$, then $D_\tau-D_\sigma\le i(N+i-1)<i(N+j-1)$, so
$c(D_\tau-D_{\sigma})<i$.

Since $\tau\otimes
S^{m}\h^*$ contains $\tau_j\otimes
S^{m-j}\h^*$, which in turn contains $\tau_{m}$, we see that the submodule
generated by the singular vector $u$
contains the copy of $\tau_{m}$ in degree $m$, which implies
that $f_{\tau,m}$ is divisible by $f_{\tau,m-1}$.

Thus, to complete the induction step,
it suffices to show that
$$
f_{\tau,m}'(0)=f_{\tau,m-1}'(0)-N-m+1.
$$
To prove this formula, let us differentiate the equation of
Proposition \ref{recfor} with respect to $c$ at $c=0$.
We get
$$
F_{0,\tau,m}'(a_1...a_mv)=
$$
$$
\frac{1}{m}\sum_{j=1}^m\left(a_jF_{0,\tau,m-1}'(a_1...a_{j-1}a_{j+1}...a_mv)-
\sum_{s\in
S} [a_1...a_m,s]v\right).
$$
This can be rewritten, using tensor notation, as follows:
$$
F_{0,\tau,m}'=\frac{1}{m}\sum_{j=1}^m(F_{0,\tau,m-1}')_{\hat j}
-\frac{1}{m}(D_\tau-D_{\tau_m}),
$$
where the subscript $\hat j$ means that the operator acts in all
components of the tensor product but the $j$-th.
Since $\tau_m\subset \tau_{m-1}\otimes \h^*\subset \tau\otimes
S^m\h^*$, this equation implies
$$
f_{\tau,m}'(0)=f_{\tau,m-1}'(0)-\frac{1}{m}(D_\tau-D_{\tau_m})=
f_{\tau,m-1}'(0)-N-m+1,
$$
as desired.
\end{proof}

Now the theorem follows easily from Lemma \ref{fprod}.
Namely, we see that $L_c(\tau)$ is not unitary
on the interval $I_k$ because the polynomial
$f_{\tau,k+1}(c)$ is negative on this interval,
and hence the form $\beta_{c,\tau}$ is negative definite on
$\tau_{k+1}$.
\end{proof}

\begin{remark}
It follows from \cite{GGOR}
that a module $L_c(\tau)$ is thin (i.e., is killed by the KZ functor or,
equivalently, has support strictly smaller than $\h$) if and only
if $\tau$ is not $m$-regular, where $m$ is the denominator of $c$
(i.e., it contains some part at least $m$ times).
On the other hand, it is easy to show directly by looking at
Young diagrams that if $\tau$ is not $m$-regular,
and $m\ge N(\tau)$, then $\tau$ is a rectangular
diagram, $\tau=(p,...,p)$, and $m=N(\tau)=p$.
Thus Theorem \ref{conta} implies that the representations of
Proposition \ref{thinuni} are the only thin unitary
representations for $c>0$ (and a similar statement is valid for
$c<0$).
\end{remark}

The following result is a special case of Theorem \ref{genera}, but was known before
Theorem \ref{genera} was proved; here we give its original proof.

\begin{theorem}\label{taum}
Let $m=m_*(\tau)$, and $\tau_m$ be the diagram
obtained from $\tau$ by removing the last column and
concatenating it with the first one (as in the proof of Theorem \ref{conta}).
Then $\frac{1}{\ell(\tau)}\in U(\tau_m)$.
In particular, Theorem \ref{genera} holds
if the multiplicity $p=p_*(\tau)$ of the part $1$ in $\tau$
satisfies the inequality $p\ge m$.
\end{theorem}

\begin{proof}
Since $[-\frac{1}{\ell(\tau)},\frac{1}{\ell(\tau)}]\subset
U(\tau)$, and (as was shown in the proof of Theorem \ref{conta})
$M_c(\tau)$ contains a singular vector in $\tau_m$ at
$c=1/\ell(\tau)$, the theorem follows from Proposition \ref{zer}(iv).
\end{proof}

\subsection{The Dunkl-Kasatani conjecture}

The following theorem was conjectured
(and partially proved) by Dunkl
(see the end of \cite{Du}) in 2005. It is also the rational
version of Kasatani's conjecture for double affine Hecke algebras
(\cite{Ka}, Conjecture 6.4), which was proposed at approximately the same time.
In the case when $c\notin \frac{1}{2}+\Bbb Z$, this theorem
was proved by Enomoto \cite{En} in 2006, using the results of Rouquier
on the connection between rational Cherednik algebras and q-Schur algebras,
andthe theory of crystal bases for quantum affine algebras.
Enomoto also explained that this theorem implies Kasatani's
conjecture. We give a different proof of this theorem,
based on the work \cite{BE}, which does not need the
condition $c\notin \frac{1}{2}+\Bbb Z$.

\begin{theorem} \label{length}
(i) Assume that $c=r/m$, where $r\ge 1, m\ge 2$ are integers with
$(r,m)=1$. Then the module $M_c(\Bbb C)$ has length
$l+1$, where $l=[n/m]$.
Namely, it has a strictly increasing filtration by submodules,
$$
0=I_c^0\subset I_c^1\subset...\subset I_c^{l+1}=M_c(\Bbb C),
$$
such that the successive quotients are irreducible.
In particular, $I_c^1=N_c$.

(ii) For $1\le j\le l+1$,
$I_c^j$ is a lowest weight representation, and its
lowest weight is the representation of the symmetric group corresponding to the
partition $\tau_c^j=(jm-1,m-1,...,m-1,s_j)$, where $n-(j-1)m=q_j(m-1)+s_j$,
$0\le s_j<m-1$ if $j\le l$, and $\tau_c^{l+1}=(n)$.

(iii) The variety $V(I_c^j)\subset \Bbb C^n$ defined by the ideal
$I_c^j, j=0,...,l+1$ is the variety $X_m^j$ of all vectors $(x_1,...,x_n)$
which, up to a permutation, have the form
$$
(x_1,...,x_{n-jm},a_1,...,a_1,a_2,...,a_2,...,a_j,...,a_j),
$$
where each $a_i$ is repeated $m$ times (here we agree that
$X_m^{l+1}=\emptyset$).

(iv) $I_c^j$ are radical ideals if and only if $r=1$.

(v) At the point $c=r/m$, the forms $\beta$ and $\gamma$ have a zero of order
exactly $l-j+1$ on the ideal $I_c^j$ for $j=1,...,l+1$.
\end{theorem}

\begin{remark}
The ideals of Theorem \ref{chercon} are
rational limits of the ideals defined in \cite{FJMM},
see also \cite{Ka}.
\end{remark}

\begin{proof}
Let us first construct the ideals $I_c^j$.
Assume first that $c=1/m$ (i.e., $r=1$).
In this case, define $I_c^j$ to be the defining ideals
of the varieties $X_m^j$. We claim that these ideals
are invariant under the Dunkl operators $D_i$, i.e. are submodules
under the rational Cherednik algebra. To check this, let
$f\in I_c^j$, and $U$ be the formal neighborhood in $\Bbb C^n$
of the $S_n$-orbit of a generic point $u\in X_m^j$. It is sufficient
to show that $D_if=0$ on the intersection $X_m^j\cap U$.
But this follows easily (using the ideology of \cite{BE}) from
the fact that the irreducible representation $L_c(S_m,\Bbb C^{m-1},\Bbb C)$
is 1-dimensional (so that the ideal of zero is a subrepresentation
of the polynomial representation $M_c(S_m,\Bbb C^{m-1},\Bbb C)$).

The case of $c=r/m$ for a general $r$ such that $(r,m)=1$ is slightly more
complicated but similar. Namely, let
$I_{r,m}$ be the maximal proper subrepresentation
in the polynomial representation $M_{r/m}(S_m,\Bbb C^m,\Bbb C)$.
We define $I_c^j$ to be the intersection of the $S_n$-images
of the ideal $\Bbb C[x_1,...,x_{n-jm}]\otimes I_{r,m}^{\otimes j}$.
Then the same argument as above shows that $I_c^j$ are
a nested sequence of subrepresentations of the polynomial representation $M_c(\Bbb C)$.
Moreover, since the representation $L_c(S_m,\Bbb C^{m-1},\Bbb C)$
is finite dimensional,
the variety defined by the ideal $I_c^j$ is $X_m^j$.
Also, it is easy to see from the definition that $I_c^j$ is
a radical ideal if and only if $r=1$. Thus, we have proved (iii) and (iv).

To prove the rest of the theorem, we need the following lemma.

\begin{lemma}\label{lenpol} The length of the polynomial representation
$M_c(\Bbb C)$ is $l+1$, and its composition factors are $L_c(\tau_c^j)$,
$j=1,...,l+1$.
\end{lemma}

\begin{proof}
It is shown in \cite{Du} that $M_c(\Bbb C)$ has singular vectors
living in $\tau_c^j$, so these irreducible representations do occur
in the composition series, so that the length of $M_c(\Bbb C)$ is at least $l+1$.

We prove that the length is in fact exactly $l+1$ (i.e. no other composition factors occur)
by induction in $n$. The base of induction is trivial, so we only need to justify the inductive step.
For this purpose, let $b\in \Bbb C^n$ be a point with stabilizer
$S_{n-1}$, and consider the restriction functor ${\rm Res}_b: {\mathcal O_c}(S_n,\Bbb C^n)\to
{\mathcal O}_c(S_{n-1},\Bbb C^n)$ defined in \cite{BE}. This functor is exact.
Moreover, the support of any simple object in ${\mathcal O}_c(S_n,\Bbb C^n)$
is $X_m^j$ for some $j$, $0\le j\le l$, so if $n$ is not divisible by $m$,
the functor ${\rm Res}_b$ does not kill any simple objects
(as $b\in X_m^j$ for all $j$ in this case). This implies that
$$
{\rm length}(M_c(S_n,\Bbb C^n,\Bbb C))\le
{\rm length}({\rm Res}_b(M_c(S_n,\Bbb C^n,\Bbb C)))=
{\rm length}(M_c(S_{n-1},\Bbb C^n,\Bbb C)).
$$
But ${\rm length}(M_c(S_{n-1},\Bbb C^n,\Bbb C))=l+1$ by the induction
assumption, so we are done.

It remains to consider the case when $n=ml$.
In this case, we have ${\rm length}(M_c(S_{n-1},\Bbb C^n,\Bbb C))=l$,
which is even better for us, but the problem is that now the functor ${\rm Res}_b$ may
(and actually does) kill simple objects, as $b\notin X_m^l$. However,
we still have $b\in X_m^j$, $j<l$, so the above argument shows that
the composition series of $M_c(S_n,\Bbb C^n,\Bbb C)$ is as desired, plus possibly some
additional simple modules supported on the variety $X_m^l$ (the smallest of all $X_m^j$).
So to prove the induction step (i.e. show that in fact there is no additional modules), it suffices
to show that the composition series of $M_c(S_n,\Bbb C^n,\Bbb C)$ contains at most one
simple module supported on $X_m^l$.

To do so, consider a point $b\in \Bbb C^n$ with stabilizer
$(S_m)^l$, and the corresponding functor
${\rm Res}_b: \
{\mathcal O_c}(S_n,\Bbb C^n)\to
{\mathcal O}_c((S_m)^l,\Bbb C^n)$.
This functor is exact and does not kill
any simple objects, as $b\in X_m^j$ for all $j$,
$0\le j\le l$. Thus, it suffices to show
that in the composition series of
${\rm Res}_b(M_c(S_n,\Bbb cn,\Bbb C))$
(or, equivalently, of
$M_c((S_m)^l,\Bbb C^n,\Bbb C)=
M_c(S_m,\Bbb C^m,\Bbb C)^{\otimes l}$)
there is at most one simple object
with support of dimension $l$.
So it is enough to check that
in the composition series of $M_c(S_m,\Bbb C^m,\Bbb C)$
there is at most one simple object with support of dimension $1$
(and all other terms have supports of larger dimension),
i.e. that in the composition series of $M_c(S_m,\Bbb C^{m-1},\Bbb C)$
there is at most one finite dimensional simple object.
But it is well known (see \cite{BEG}) that
this composition series involves only two simple objects,
only one of which is finite dimenional.
\end{proof}

Now we finish the proof of Theorem \ref{length}.
Lemma \ref{lenpol} implies that the quotient
$I_c^{j+1}/I_c^j$ is irreducible for each $j$.
Also, the support of this representation is $X_m^j$.
So since by \cite{BE}, the support of $L_c(\tau_c^j)$ is also $X_m^j$, we find that
$I_c^{j+1}/I_c^j=L_c(\tau_c^j)$.

\begin{lemma}\label{submo} Any submodule $E$ of $M_c(\Bbb C)$ coincides with
$I_c^j$ for some $j$.
\end{lemma}

\begin{proof}
Let $X$ be the variety defined by $E$.
Then $X=X_m^j$ for some $j$ (by \cite{BE}, Section 3.8).
So $M_c(\Bbb C)/E$ may involve in its Jordan-Holder series
only $L_c(\tau_c^i)$ with $i>j$. This means that $E\supset
I_c^j$. On the other hand, restricting to a generic point of
$X_m^j$ (as in \cite{BE}), we see that we must have $E\subset
I_c^j$. This implies $E=I_c^j$, as desired.
\end{proof}

Finally, recall again from \cite{Du} that
$M_c(\Bbb C)$ contains singular vectors in representations
$\tau_c^j$, $j=1,...,l+1$. Let $W_j$ be the highest
weight submodules of $M_c(\Bbb C)$ generated by
these singular vectors. The unique irreducible quotient of $W_j$ is
$L_c(\tau_c^j)$, so by the above we must have $W_j=I_c^j$.
In particular, $I_c^j=W_j$ are a nested sequence of lowest
weight modules, as anticipated in \cite{Du}, Section 6.

We have now established (i) and (ii).
To establish the remaining statement (v),
it suffices to observe that it follows from (i)-(iv) that
the Jantzen filtration on $M_c(\Bbb C)$ coincides with the
filtration by the ideals $I_c^j$.\footnote{This fact is discussed
in \cite{Ch2}, p.15, and also in \cite{Ch1}.}
\end{proof}

\subsection{Unitarity of the irreducible subrepresentation
of the polynomial representation}

\begin{theorem}\label{chercon}
Let $W=S_n$, $\h=\Bbb C^n$, and $2\le m\le n$.
Then $N_{1/m}\subset \Bbb C[x_1,...,x_n]$
is contained in $L^2(\h_{\Bbb R},d\mu_c)$.
In particular, $N_{1/m}$ is unitary.
Thus, the answer to
Cherednik's Question \ref{che2} for type A is affirmative.
\end{theorem}

\begin{remark}
Note that the statement that $N_{1/m}$ is unitary in Theorem
\ref{chercon} is a special case of Theorem \ref{genera} (taking into
account Theorem \ref{length}(ii)).
\end{remark}

\begin{proof}
Let $P\in \Bbb C[x_1,...,x_n]$, and
set
$$
\xi_P(c)=\int_{\Bbb R^n}|P(z)|^2d\mu_c(z).
$$
It is a standard fact that $\xi_P$ is a holomorphic function
of $c$ for ${\rm Re}c\le 0$ which extends meromorphically
to the whole complex plane. By Proposition \ref{inte},
$\xi_P(c)=K(c)\gamma_{c,\Bbb C}(P,P)$, which implies that the
poles of $\xi_P(c)$ may occur only for $c=r/m>0$, where $2\le
m\le n$ and $(r,m)=1$, and the order of such a pole is
at most $[n/m]+1$ (which is the order of the pole of $K(c)$ at
$c=r/m$). In fact, it is clear from Theorem \ref{length}(v) that
the order of the pole of $\xi_P(c)$ at $c=r/m$ is
at most $j-1$ if $P\in I_c^j$, $j>0$. In particular,
there is no pole of $\xi_c(P)$ for $c=r/m$ if $P\in N_{r/m}$.

The proof is based on the following proposition.

\begin{proposition}\label{cont}
If $P\in N_{1/m}$, $2\le m\le n$, then $\xi_P(c)$ has no poles
for $c<\frac{1}{m-1}$.
\end{proposition}

This proposition implies Theorem \ref{chercon}.
Indeed, it implies that for \linebreak $c<\frac{1}{m-1}$,
$P\in L^2(\Bbb R^n,d\mu_c)$.

\begin{proof} (of Proposition
\ref{cont}).

Let $\frac{r}{k}<\frac{1}{m-1}$, $(r,k)=1$, so $r(m-1)<k\le n$.
Our job is to show that $N_{1/m}\subset N_{r/k}$, so that
$\xi_P(c)$ has no pole at $c=r/k$.

Let $S_c$ be the scheme defined by the ideal $N_c$.
We have to show that $S_{r/k}\subset S_{1/m}$.

By Theorem \ref{length}(iii,iv),
$S_{1/m}$ is the reduced scheme (variety)
$X_m^1$, which is the set of all points
$(x_1,...,x_n)$ such that
some $m$ coordinates coincide with each other.
The scheme $S_{r/k}$ is not necessarily reduced, but
by Theorem \ref{length}(iii),
the underlying variety $\bar S_{r/k}$
is $X_k^1$.

By Theorem \ref{length}(i),
$N_{r/k}$ is the set of all $f\in \Bbb C[x_1,...,x_n]$
such that $f$ vanishes in the formal neighborhood in $S_{r/k}$ of
a generic point of $X_k^1$, i.e. a point $u=(x_1,...,x_n)$
where some $k$ coordinates coincide with each other, and there is
no other coincidences. Therefore,
it suffices to check that for
any $f\in N_{1/m}$, $f$ vanishes on the formal neighborhood
in $S_{r/k}$ of $u$.
For this, it suffices to check that this holds if $f$ belongs to
the lowest weight subspace $Q$ of $N_{1/m}$.
By using \cite{BE} and restricting to the formal neighborhood
of $u$, we see that it is sufficient to show that the
representation $Q$, regarded as a representation
of $S_k$, is disjoint from
the representation $L_{r/k}(S_k,\Bbb C^{k-1},\Bbb C)$ regarded as an
$S_k$-module (i.e., there is no nontrivial homomorphisms
between these two representations).

Now, we know from Theorem \ref{length}(ii) that the lowest weight
subspace $Q$ is the representation of $S_n$ corresponding to the
Young diagram $\tau_m^1=(m-1,...,m-1,s)$, where
$n=q(m-1)+s$. Also recall from
\cite{BEG} that $L_{r/k}(S_k,\Bbb C^{k-1},\Bbb C)$,
as a representation of $S_k$, is
the space of complex functions on the group $A\subset (\Bbb Z/r\Bbb Z)^k$
of vectors with zero sum of coordinates. Such a vector has
a coordinate $i\in \Bbb Z/r\Bbb Z$ with multiplicity $n_i$,
and $\sum n_i=k$. So irreducible representations of $S_n$
that occur in $L_{r/k}(S_k,\Bbb C^{k-1},\Bbb C)$
are those representations $Y$ for which
$Y^{S_{n_1}\times...\times S_{n_r}}\ne 0$ for some
$n_1,...,n_r$ such that $n_1+...+n_r=k$.

However, we claim that $Q^{S_{n_1}\times...\times S_{n_r}}=0$
for any $n_1,...,n_r$ with $n_1+...+n_r=k$.
Indeed, it is standard that $Q$ is generated by the polynomial
$P:=\Delta_{q+1}^{\otimes s}\otimes \Delta_q^{\otimes m-1-s}$, where
$\Delta_p$ is the Vandermonde determinant in $p$ variables.
On the other hand, since
$r(m-1)<k$, we have $m-1<n_i$ for some $i$.
This means that if we symmetrize $P$ with respect to any
conjugate of the subgroup $S_{n_i}$, we get zero, as desired.
\end{proof}

Theorem \ref{chercon} is proved.
\end{proof}

\section{Appendix: Proof of Theorem \ref{genera}}

\centerline{\large Stephen Griffeth}
\vskip .13in

The purpose of this appendix is to apply the results in Suzuki's paper 
\cite{Su}, which classifies and describes those irreducible modules in 
category $\OO$ on which the Cherednik-Dunkl subalgebra acts diagonalizably,
to the problem of unitarity of $L_c(\tau)$.  In Theorem~3.7.2 of his book 
\cite{Ch3}, I. Cherednik proved results analogous to Suzuki's for the double 
affine Hecke algebra of type $A$ and it would be interesting to apply them to 
classify the unitary modules for the double affine Hecke algebra.  

We will use the definitions and notation of sections 1-5 of the present paper, 
except as noted in this paragraph.  In order to conform with Suzuki's
notation, we set $\kappa=-1/c$, and write $\HH_\kappa$ for the rational
Cherednik algebra attached to $S_n$ acting on its permutation
representation.  Let $y_1,\dots,y_n$ be the standard basis of the 
permutation representation $\hh=\CC^n$ of $S_n$ and let $x_1,\dots,x_n$ be 
the dual basis of $\hh^*$.  As in \cite{Su} the commutation relation 
for $y_i$ and $x_j$ is
\begin{equation}
y_i x_j=\begin{cases} x_j y_i-s_{ij} \quad &\hbox{if $i \neq j$, and} \\
x_i y_i+\kappa+\sum_{k \neq i} s_{ik} \quad &\hbox{if $i=j$.} \end{cases}
\end{equation}  This differs from the relations used in Section 5 of this
paper only by a scaling that does not affect the question of unitarity.  We
write $L_\kappa(\tau)$ for the irreducible representation corresponding
to a partition $\tau$.  For the remainder of the paper we assume that
$\kappa \in \QQ$.

The \emph{Cherednik-Dunkl subalgebra} of $\HH_\kappa$ is generated by the
elements
\begin{equation}
z_i=y_i x_i-\phi_i \quad \text{where} \quad \phi_i=\sum_{1 \leq j < i} s_{ij}
\end{equation}  The elements
$z_1,\dots z_n$ are pairwise commutative.  Let $w_0 \in S_n$ be the longest
element, with $w_0(i)=n-i+1$ for $1 \leq i \leq n$.  By the defining
relations for $\HH_\kappa$,
\begin{align*}
w_0 z_i w_0^{-1}&=y_{n-i+1} x_{n-i+1}-\sum_{1 \leq j < i } s_{n-i+1,n-j+1} \\
&=x_{n-i+1} y_{n-i+1}+\kappa+\sum_{j \neq n-i+1} s_{n-i+1,j}-\sum_{1 \leq j < i } s_{n-i+1,n-j+1} \\
&=\epsilon_{n-i+1}^\vee+\kappa
\end{align*} where as in Proposition~4.2 of Suzuki's paper \cite{Su},
\begin{equation}
\epsilon_i^\vee=x_i y_i+\sum_{1 \leq j <i} s_{ij}
\end{equation}  Note that $z_i$ and $\epsilon_i^\vee$ are invariant with
respect to the antiautomorphism $*$, and hence act diagonalizably on the
unitary modules $L_\kappa(\tau)$.  Suzuki's paper describes the irreducible
modules in category $\OO$ for $\HH_\kappa$ on which
$\epsilon_1^\vee,\dots,\epsilon_n^\vee$ (or equivalently $z_1,\dots,z_n$) act
diagonalizably and gives an explicit combinatorial description of their
weight space decompositions.

We will need the \emph{intertwining operators }
\begin{equation}
\sigma_i=s_i-\frac{1}{z_i-z_{i+1}}, \ \Phi=x_n s_{n-1} \cdots s_1, \
\text{and} \ \Psi=y_1 s_1 \cdots s_{n-1},
\end{equation} the fact that they map eigenvectors for $z_1,\dots,z_n$ to
eigenvectors, and the formulas
\begin{equation}
\sigma_i^2=\frac{(z_i-z_{i+1})^2-1}{(z_i-z_{i+1})^2}, \quad \Psi \Phi=z_1,
\end{equation}
\begin{equation}
\sigma_i^*=\sigma_i \quad \text{and} \quad \Phi^*=\Psi
\end{equation} where $*$ is the conjugate linear anti-automorphism 
of $\HH_\kappa$ with $x^*=Tx$, $y^*=Ty$, and $w^*=w^{-1}$
for $x \in \hh$, $y \in \hh$, and $w \in S_n$.  For any 
$f \in \HH_\kappa$, $f^*$ is the adjoint of $f$ with respect to the 
contravariant form on $L_\kappa(\tau)$.  All of this is by now standard;
proofs of the more general facts for the groups $G(r,1,n)$ may be found in
\cite{Gri1}.   

Let $\tau$ be a partition of $n$ of length $m$ and let $\kappa \in \QQ_{>0}$.
We regard $\tau$ as the subset of $\ZZ \times \QQ$ containing the points
$(i,j)$ for $i,j \in \ZZ$ with $1 \leq i \leq m$ and $1 \leq j \leq \tau_i$.  
Put $p=(-m,\kappa-m)$ and let $\widehat{\tau}$ be the subset of 
$\ZZ \times \QQ$ given by
\begin{equation}
\widehat{\tau}=\tau+\ZZ p.
\end{equation}  It contains $\tau$ as a subset.  A \emph{periodic tableau} on 
$\widehat{\tau}$ is a bijection $T:\widehat{\tau} \rightarrow \ZZ$ such that
\begin{equation}
\hbox{for all $b \in \widehat{\tau}$, we have $T(b+p)=T(b)-n$}
\end{equation} and a periodic tableau $T$ is \emph{standard} if for 
$(a,b) \in \widehat{\tau}$ such that $(a,b+1) \in \widehat{\tau}$, we have
\begin{equation}
T(a,b) < T(a,b+1)
\end{equation} and for $(a,b) \in \widehat{\tau}$ and $k \in \ZZ_{\geq 0}$ with 
$(a+k+1,b+k) \in \widehat{\tau}$ we have
\begin{equation}
T(a,b) < T(a+k+1,b+k).
\end{equation}  The \emph{content vector} of the tableau $T$ is
\begin{equation}
\text{ct}(T)=\left(\text{ct}(T^{-1}(1)),\dots,\text{ct}(T^{-1}(n)) \right), 
\quad \hbox{where $\text{ct}(a,b)=b-a$.}
\end{equation}  When $\kappa \in \ZZ$ we
will think of $\tau$ and $\widehat{\tau}$ as consisting of collections of 
boxes, and a tableau as a filling of those boxes with integers.

Having fixed $n \in \ZZ_{>0}$, write $P$ for the set of integer partitions 
of $n$.  Let $\kappa \in \QQ_{>0}$, and write $\kappa=s/r$ with 
$r,s \in \ZZ_{>0}$ and $(r,s)=1$.  Recall the definition of $N(\tau)$ from 
the second paragraph of Section~\ref{s1} and define
\begin{equation}
P_\kappa=\{\tau \in P \ | \ s \geq N(\tau^*)\} \quad
\hbox{where $\tau^*$ is the transpose of $\tau$.}
\end{equation}  Write
\begin{equation}
S(\widehat{\tau})=\{\hbox{standard tableaux $T$ on
$\widehat{\tau}$ with $T(b)>0$ for $b \in \tau$} \}
\end{equation}  Now we combine Theorems 4.8 and 4.12 of Suzuki's paper 
\cite{Su} into the following theorem:
\begin{theorem}[Suzuki] \label{SuzThm}
The set of diagonalizable irreducible $\HH_\kappa$-modules in $\OO$ is
$\{L_\kappa(\tau) \ | \ \tau \in P_\kappa\}$, and for
$\tau \in P_\kappa$ we have
\begin{equation*}
L_\kappa(\tau)=\bigoplus_{T \in S(\widehat{\tau})} L_\kappa(\tau)_{\text{ct}(T)},
\quad \text{with} \quad
\text{dim}_\CC( L_\kappa(\tau)_{\text{ct}(T)})=1 \
\hbox{for all $T \in S(\widehat{\tau})$,}
\end{equation*} where for a sequence $a_1,\dots,a_n$ of numbers
\begin{equation*}
L_\kappa(\tau)_{(a_1,\dots,a_n)}=
\{m \in L(\tau) \ | \ \epsilon_i^\vee.m=a_i m \ \hbox{for $1 \leq i \leq n$} \}
\end{equation*}
\end{theorem}

We now finish the proof of Theorem \ref{genera} started in Section 5. 
We rewrite the statement of Theorem \ref{genera} in terms 
of $\kappa=-1/c$ so that it becomes:
\begin{align*} 
\{\kappa \ | \ L_\kappa(\tau) \ \text{is unitary} \}&= 
\{\kappa \in \ZZ_{>0} \ | \ \kappa \geq N(\tau^*) \} \\
&\cup \{\kappa \in \ZZ_{<0} \ |
\ \kappa \leq -N(\tau) \} \\
&\cup \{\kappa \in \QQ \ | \ |\kappa| \geq \ell(\tau) \}.
\end{align*}  
This will follow from the results of Section 5 and 
Theorem~\ref{Unitary theorem} below.  First we prove a preparatory lemma.

Recall that $\ttt$ is the Cherednik-Dunkl subalgebra generated by 
$z_1,\dots,z_n$.  A weight of $\ttt$ on a module $M$ is thus described 
by the sequence $\alpha=(\alpha_1,\dots,\alpha_n)$ of complex numbers 
giving the action of $z_1,\dots,z_n$.

\begin{lemma} \label{numerical crit}
The module $L_\kappa(\tau)$ is unitary exactly if it is 
$\ttt$-diagonalizable and
for all weights $\alpha=(\alpha_1,\dots,\alpha_n)$ of $\ttt$ on 
$L_\kappa(\tau)$,
\begin{equation} \label{weight conditions}
\alpha_1 \geq 0 \quad \text{and} \quad (\alpha_i-\alpha_{i+1})^2 \geq 1
\end{equation}
\end{lemma}
\begin{proof}
The module $L_\kappa(\tau)$ is graded by finite dimensional subspaces, 
and $\ttt$ preserves these subspaces.  The contravariant form is unitary 
on $L_\kappa(\tau)$ exactly if its restriction to each graded piece is.

Suppose first that $L_\kappa(\tau)$ is unitary.  Then $z_1,\dots,z_n$ 
act on each graded piece as commuting self-adjoint operators, and it 
follows that $L_\kappa(\tau)$ is $\ttt$-diagonalizable.  For a 
$\ttt$-eigenvector $f \in L_\kappa(\tau)$ of weight $\alpha$, one has
\begin{equation} \label{ineq1}
 \la \Phi.f,\Phi.f \ra=\la f,\Psi \Phi.f \ra=\la f, z_1.f \ra
=\alpha_1 \la f,f \ra
\end{equation} and it follows that $\alpha_1 \geq 0$.  Similarly, for $1 \leq i \leq n-1$
\begin{equation} \label{ineq2}
 \la \sigma_i.f,\sigma_i.f \ra =\la f, \sigma_i^2.f \ra 
=\frac{(\alpha_i-\alpha_{i+1})^2-1}{(\alpha_i-\alpha_{i+1})^2} \la f,f \ra
\end{equation} and it follows that $(\alpha_i-\alpha_{i+1})^2 \geq 1$.

Conversely, the irreducibility of $L_\kappa(\tau)$ implies that each 
weight vector of $\ttt$ can be obtained from a weight vector in 
$S^\tau \subseteq L_\kappa(\tau)$ by applying a sequence of operators 
from the set $\{\Phi,\sigma_1,\dots,\sigma_{n-1}\}$.  Assuming that 
$L_\kappa(\tau)$ is diagonalizable and that \eqref{weight conditions} 
holds, the equalities \eqref{ineq1} and \eqref{ineq2} imply that the 
$\ttt$-eigenvectors in $L(\tau)$ have non-negative norms.  Therefore 
Suzuki's theorem and the orthogonality of distinct weight spaces 
implies that $L_\kappa(\tau)$ is unitary.
\end{proof}

The lemma reduces the question of unitarity of irreducible lowest 
weight modules for the rational Cherednik algebra of type $S_n$ to 
an examination of their $\ttt$-spectra.  One could prove Theorem~5.6 
this way.  We will use it to finish the proof of the converse:
\begin{theorem} \label{Unitary theorem}
If $\kappa \in \ZZ$ and $\kappa \geq N(\tau^*)$ or 
$\kappa \leq -N(\tau)$ then $L_\kappa(\tau)$ is unitary.
\end{theorem}
\begin{proof}
We will show using Lemma~\ref{numerical crit} that if $\kappa \in \ZZ$ 
and $\kappa \geq N(\tau^*)$ then $L_\kappa(\tau)$ is unitary.  
Proposition 3.3 then completes the proof.  Since $\kappa \geq N(\tau^*)$, 
Suzuki's Theorem~\ref{SuzThm} implies that the module $L_\kappa(\tau)$ 
is diagonalizable with weight basis $f_T$ (for $\ttt$) in bijection 
with the set $S(\widehat{\tau})$ of standard tableaux $T$ on 
$\hat{\tau}=\tau+\ZZ (-m,\kappa-m)$ such that $T(b)>0$ for $b \in \tau$,
where $m$ is the length of $\tau$.  The weight of $f_T$ is given by
\begin{equation}
z_i.f_T=(\text{ct}(T^{-1}(n-i+1))+\kappa) f_T 
\quad \hbox{for $1 \leq i \leq n$.}
\end{equation}

The contents of the boxes of $\widehat{\tau}$ are all integers since 
$\kappa$ is an integer.  This implies the second inequality in 
\eqref{weight conditions} of Lemma~\ref{numerical crit}.  Furthermore, 
adding $\kappa$ to the content of the box containing $n$ gives a 
non-negative integer by the definitions of $N(\tau^*)$ and the set 
$S(\widehat{\tau})$.  This implies the first 
inequality of \eqref{weight conditions} of Lemma~\ref{numerical crit}.  
The theorem is proved.
\end{proof}

By using the same techniques and the results of \cite{Gri2}, one 
should be able to determine the set of unitary lowest weight irreducibles 
for rational Cherednik algebras attached to the groups $G(r,p,n)$.  The 
missing ingredient is the analog of Theorem~\ref{SuzThm} for the groups 
$G(r,1,n)$.  We are currently working on this problem, using the results 
of \cite{Gri2} as a starting point.

\end{document}